\documentclass[11pt,reqno, a4paper]{amsart}
\usepackage{amsfonts}
\usepackage{amssymb}
\usepackage{txfonts}
\usepackage{mathrsfs}
\usepackage{bbm}
\usepackage{enumerate}

\newtheorem{thm}{Theorem}[section]
 
 \newtheorem{lem}[thm]{Lemma}
 \newtheorem{prop}[thm]{Proposition}
 
 \newtheorem{rmk}[thm]{Remark}
 \numberwithin{equation}{section}

 \newcommand{\vpi}{\varphi}
 \newcommand{\vstar}{v_{*}}

 \newcommand{\p}{\partial}
 
 \newcommand{\s}{\mathcal{S}}
 \newcommand{\Real}{\mathbb{R}}
 
 \newcommand{\Natural}{\mathbb{N}}
 
 \newcommand{\norm}[1]{\Vert#1\Vert}
 
 \newcommand{\abs}[1]{\left\vert#1\right\vert}
 \newcommand{\bigabs}[1]{\big\vert#1\big\vert}
 \newcommand{\set}[1]{\left\{#1\right\}}
 \newcommand{\bigset}[1]{\big\{#1\big\}}
 \newcommand{\inner}[1]{\left(#1\right)}
 \newcommand{\biginner}[1]{\big(#1\big)}
 
 \newcommand{\com}[1]{\big[#1\big]}
 \newcommand{\reff}[1]{(\ref{#1})}
 \newcommand{\V}{\,\,|\,\,\,}

\begin{document}


\title[Analytic smoothness effect for  Landau
equation]{Analytic smoothness effect of solutions for spatially
homogeneous Landau equation} \thanks{This work is partially
supported by the NSFC.}

\author{Hua Chen \and Wei-Xi LI \and Chao-Jiang Xu}
\address{School of Mathematics and Statistics, Wuhan University, 430072 Wuhan, China} \email{chenhua@whu.edu.cn}
\address{School of Mathematics and Statistics, Wuhan University, 430072 Wuhan, China} \email{wei-xi.li@whu.edu.cn}
\address{School of Mathematics and Statistics, Wuhan University, 430072 Wuhan, China
\newline\indent And
\newline\indent Universit\'{e} de Rouen, UMR 6085-CNRS, Math\'{e}matiques
Avenue de l'Universit\'{e}, BP.12, 76801 Saint Etienne du Rouvray,
France } \email{Chao-Jiang.Xu@univ-rouen.fr}
%
%
%
%

\begin{abstract}
In this paper, we study the smoothness effect of Cauchy problem for
the spatially homogeneous Landau equation in the hard potential case
and the Maxwellian molecules case. We obtain the analytic smoothing
effect for the solutions under rather weak assumptions on the
initial datum.
\end{abstract}
%

%

 \maketitle

\section{Introduction}
\label{}

In this paper we study the Cauchy problem for the following
spatially homogeneous Landau equation
\begin{equation}\label{Landau}
\left\{
\begin{array}{ll}
   \partial_t f=\nabla_{v}\cdot\bigset{\int_{\Real^3}
   a(v-\vstar)[f(\vstar)\nabla_v
   f(v)-f(v)\nabla_{v}f(\vstar)]d\vstar},\\
  f(0,v)=f_0(v),
\end{array}
\right.
\end{equation}
where $f(t,v)\geq 0$ stands for the density of particles with
velocity $v\in\Real^3$ at time $t\geq0$, and $(a_{ij})$ is a
nonnegative symmetric matrix given by
\begin{equation}\label{coe}
  a_{ij}(v)=\inner{\delta_{ij}-\frac{v_iv_j}{\abs v^2}}\abs
  v^{\gamma+2}.
\end{equation}
We only consider here the condition $\gamma\in[0, 1]$, which is
called the hard potential case when $\gamma\in(0, 1]$ and the
Maxwellian molecules case when $\gamma=0$.  Set
$c=\sum\limits_{i,j=1}^3\partial_{v_iv_j}a_{ij}=-2(\gamma+3)\abs
 v^\gamma$ and
 \[
 \bar a_{ij}(t,v)=\inner{a_{ij}*f}(t,v)=\int_{\Real^3}a_{ij}(v-\vstar)f(t,\vstar)d\vstar,
 \quad \bar c=c*f.
 \]
Then the Cauchy problem \reff{Landau} can be rewritten as
\begin{equation}\label{Landau'}
\left\{
  \begin{array}{ll}
  \partial_t f=\sum\limits_{i,j=1}^3
  \bar a_{ij}\partial_{v_iv_j}f-\bar cf,\\
  f(0,v)=f_0(v),
  \end{array}
\right.
\end{equation}
which is a non-linear  diffusion equation,  for  the coefficients
$\bar a_{i,j}$ satisfy some uniformly elliptic property (see
Proposition 4 in \cite{desv-vill1}) and  $\bar a_{ij}, \bar c$
depend on the solution $f$.

The Landau equation can be obtained as a limit of the Boltzmann
equation when the collisions become grazing, cf. \cite{des1} and
references therein for more details. In this paper, we are mainly
concerned with the analytic regularity of the solutions for the
spatially homogeneous Landau equation, which gives partial support
to the conjecture on the smoothness of solutions for the Boltzmann
equation with singular (or non cutoff) cross sections. This
conjecture has been proven in certain particular cases, see
\cite{al-2,aumxy, desv-wen1} for the Sobolev smoothness and
\cite{MUXY1} for the Gevrey smoothness. Recently, a lot of progress
has been obtained on the study of the Sobolev regularity for the
solutions of Landau equations, cf \cite{C-D-H, desv-vill1, Guo,
villani, villani2} and references therein, which shows that in some
sense the Landau equation can be seen as a non-linear and non-local
analog of the hypo-elliptic Fokker-Planck equation. That means, the
weak solutions, once constructed under rather weak assumptions on
the initial datum, will become smooth or, even more, rapidly
decreasing in $v$ at infinity. In the Gevrey class frame, some
results have been obtained on the propagation of regularity for the
solutions of the Landau equation or Boltzmann equation (see
\cite{CLX,des-fur-ter,ukai}).

Motivated by the smoothness effect of heat equation, one may expect
analytic or even more ultra-analytic regularity in $t>0$ for the
solutions of the Cauchy problem \reff{Landau'}. Recently,
Morimoto-Xu \cite{MXJ} proved the ultra-analytic effect for the
Cauchy problem \reff{Landau'} of the Maxwellian molecules case,
which is understandable since in this particular case the
coefficients $a_{i,j}(v)$ are ultra-analytic functions (polynomial
functions). Here we shall consider the analytic smoothness effect in
the hard potential case, which would be more complicated since in
this case the coefficients $a_{i,j}$, given in \reff{coe}, are no
longer a polynomial functions and only analytic functions away from
the origin. So it is reasonable in this paper to consider the
analytic smoothness effects of the Cauchy problem \reff{Landau'}.

Now we give some notations used throughout the paper.  For a
mult-index $\alpha=(\alpha_1,\alpha_2,\alpha_3),$ denote
\[\abs\alpha=\alpha_1+\alpha_2+\alpha_3,\quad
\alpha!=\alpha_1!\alpha_2!\alpha_3!, \quad
\p^\alpha=\p^{\alpha_1}_{v_1}\p^{\alpha_2}_{v_2}\p^{\alpha_3}_{v_3}.
\]
We say $\beta=(\beta_1,\beta_2, \beta_3)\leq (\alpha_1,\alpha_2,
\alpha_3)=\alpha$ if $\beta_i\leq \alpha_i$ for each $i$. For a
multi-index $\alpha$ and a nonnegative integer $k$ with
$k\leq\abs\alpha,$  if no confusion occurs, we shall use $\alpha-k$
to denote some multi-index $\bar\alpha$ satisfying
$\bar\alpha\leq\alpha$ and $\abs{\bar\alpha}=\abs\alpha-k $.  As in
\cite{desv-vill1}, we denote by $M(f(t)), E(f(t))$ and $H(f(t))$
respectively  the mass, energy and entropy of the function $f(t,v)$,
i.e.,
\begin{eqnarray*}
M(f(t))=\int_{\Real^3}f(t,v)\,dv,\quad
E(f(t))={1\over2}\int_{\Real^3}f(t,v)\abs v^2\,dv,
\end{eqnarray*}
\[H(f(t))=\int_{\Real^3}f(t,v)\log f(t,v)\,dv,\]
and denote $M_0=M(f(0))$, $E_0=E(f(0))$ and $H_0=H(f(0)).$ It's
known that the solutions of the Landau equation satisfy the formal
conservation laws:
\[M(f(t))=M_0 ,\quad
E(f(t))=E_0,\quad H(f(t))\leq H_0,\qquad \forall\:t\geq 0.\] Here we
adopt the following notations,
\[
 \norm {\p^\alpha f(t,\cdot)}_{L_s^p}=\inner{\int_{\Real^3}\abs{\p^\alpha f(t,v)}^p\inner{1+\abs
 v^2}^{s/2}dv}^{{1\over p}},\quad p\geq 1,
\]
\[ \norm{f(t,\cdot)}_{H^m_s}^2=\sum\limits_{\abs\alpha\leq m}\norm{\p^\alpha f(t)}_{L^2_s}^2.\]
 In the sequel, for simplicity we always write $\norm
 {f(t)}_{L_s^p}$ instead of  $\norm
 {f(t,\cdot)}_{L_s^p}$,  etc.

Before stating our main theorem, we recall some related results
obtained in \cite{desv-vill1,villani2}. In the hard potential case,
the existence, uniqueness and Sobolev regularity of the weak
solution had been studied by Desvillettes-Villani (cf. Theorem 5,
Theorem 6 and Theorem 7 of \cite{desv-vill1}), and they proved that,
under rather weak assumptions on the initial datum (e.g. $f_0\in
L^1_{2+\delta}$ with $\delta>0,$), there exists a weak solution $f$
of the Cauchy problem \reff{Landau'} such that for all time $t_0>0,$
all integer $m\geq 0,$ and all $s>0,$
\[\sup\limits_{t\geq t_0}\norm{f(t,\cdot)}_{H^m_s}\leq C,\]
where $C$ is a constant depending only on $\gamma$, $M_0$, $E_0$,
$H_0$, $m,s$ and $t_0.$ Moreover $f(t,v)\in
C^{\infty}\biginner{\Real_t^+; \s(\Real^3)},$ where $\Real_t^+=]0,
+\infty[$ and $\s(\Real^3)$ denotes the space of smooth functions
which are rapidly decreasing in $v$ at infinity.  If  $f_0\in L^2_p$
with $p>5\gamma+15,$ then the Cauchy problem \reff{Landau'} admits a
unique smooth solution.  In the Maxwellian case, Villani
\cite{villani2} proved that the Cauchy problem \reff{Landau'} admits
a unique classical solution for any initial datum and for all $t>0,$
$f(t,v)$ is bounded and belong to  $C^\infty(\Real_v^3).$

Now let us give some equivalent definition of analytic functions.
Let $u$ be a real function defined in $\Real^N.$ We say $u$ is real
analytic in $\Real^N$ if $u\in C^\infty(\Real^N)$ and there exists a
constant $C$  such that for all multi-indices $\alpha\in\Natural^N,$
\[
  \norm{\partial^\alpha{u}}_{L^2(\Real^N)}\leq C^{|\alpha|}
  \abs\alpha!,
\]
which is equivalent to
\[
 e^{c_0(-\triangle_v)^{1\over2}}u\in L^2(\Real^N)
\]
for some constant $c_0>0,$ where $e^{c_0(-\triangle_v)^{1\over2}}u$
is the Fourier multiplier defined by
\[
  e^{c_0(-\triangle_v)^{1\over2}}u=\mathscr{F}^{-1}\inner{e^{c_0\abs\xi}\hat
  u(\xi)}.
\]

Starting from  the smooth solution, we state our main result on the
analytic regularity  as follows.

\begin{thm}\label{main result}
Let $f_0$ be the initial datum with finite mass, energy and entropy
and $f(t,v)$ be any solution of the Cauchy problem \reff{Landau'}
such that for all time $t_0>0$ and all integer $m\geq 0,$
\begin{equation}\label{uniform}
  \sup\limits_{t\geq t_0}\norm{f(t,\cdot)}_{H^m_\gamma}\leq C
\end{equation}
with  $C$ a constant depending only on $\gamma$, $M_0$, $E_0$,
$H_0$, $m$ and $t_0.$ Then for all time $t>0,$ $f(t,v)$, as a real
function of $v$ variable, is analytic in $\Real^3_v.$ Moreover, for
all time $t_0>0,$ there exists a constant $c_0>0,$ depending only on
$M_0,E_0,H_0,\gamma$ and $t_0,$ such that for all $t\geq t_0,$
\[
 \norm{e^{c_0\inner{-\triangle_v}^{1\over2}}f(t,\cdot)}_{L^2(\Real^3)}\leq C(t+1),
\]
where $C$ is a constant depending only on  $M_0,E_0,H_0,\gamma$ and
$t_0.$
\end{thm}

\begin{rmk}
As a consequence, the solutions given in \cite{desv-vill1} are  real
analytic in $\Real_v^3$ for any $t>0$.
\end{rmk}
\begin{rmk}
The result of Theorem \ref{main result} can be extended to any space
dimensional case.
\end{rmk}

The plan of the paper is as follows: In section \ref{section2} we
present the proof of the main result. Section \ref{section3} is
devoted to the proof of the lemma \ref{crucial} in the section
\ref{section2} which is crucial to the proof of the main result
here.

\section{Proof of the main result}
\label{section2}

This section is devoted to the proof of  the main results.  To
simplify the notations, in the sequel we always use $\sum_{1\leq
\abs\beta\leq \abs\mu}$ to denote the summation over all the
multi-indices $\beta$ satisfying $\beta\leq\mu$ and $1\leq
\abs\beta\leq\abs\mu$. Likewise $\sum_{1\leq \abs\beta\leq
\abs\mu-1}$ denotes the summation over all the multi-indices $\beta$
satisfying $\beta\leq\mu$ and $1\leq \abs\beta\leq\abs\mu-1,$ etc.
We begin with the following lemma.

\begin{lem}\label{lemma1}
For all multi-indices $\mu\in\Natural^3, \abs\mu\geq 2,$ we have
\begin{equation}\label{sum1}
  \sum\limits_{1\leq \abs\beta\leq \abs\mu-1}\frac{\abs\mu}
  {\abs\beta^{4}(\abs\mu-\abs\beta)}\leq
 24,
\end{equation}
and
\begin{equation}\label{sum2}
  \sum\limits_{1\leq \abs\beta\leq \abs\mu-1}\frac{\abs\mu}
  {\abs\beta^{3}(\abs\mu-\abs\beta)^2}\leq 24.
\end{equation}
\end{lem}

\begin{proof}
For each positive integer $l,$ we denote by $N\set{\abs\beta=l}$ the
number of the multi-indices $\beta$ with $\abs\beta=l.$ In the case
when the space dimension equals to 3, one has
\[N\set{\abs\beta=l}=\frac{(l+2)!}{2!\,l!}={1\over2}(l+1)(l+2).\]
Thus we can compute directly
\begin{eqnarray*}
  \sum\limits_{1\leq \abs\beta\leq \abs\mu-1}\frac{\abs\mu}
  {\abs\beta^{4}(\abs\mu-\abs\beta)}\leq\sum\limits_{l=1}^{
  \abs\mu-1}\sum\limits_{\abs\beta=l}\frac{\abs\mu}
  {l^{4}(\abs\mu-l)}\leq \sum\limits_{l=1}^{
  \abs\mu-1}\frac{3\abs\mu}{l^{2}(\abs\mu-l)}.
\end{eqnarray*}
Without loss of generality, we may assume $\abs\mu-1$ is an even
integer. Then the inequality above can be rewritten as
\begin{eqnarray*}
  \sum\limits_{1\leq \abs\beta\leq \abs\mu-1}\frac{\abs\mu}
  {\abs\beta^{4}(\abs\mu-\abs\beta)}\leq\sum\limits_{l=1}^{
  (\abs\mu-1)/2}\frac{3\abs\mu}{l^{2}(\abs\mu-l)}+\sum\limits_{l=(\abs\mu+1)/2}^{
  \abs\mu-1}\frac{3\abs\mu}{l^{2}(\abs\mu-l)}.
\end{eqnarray*}
In the case when  $l\leq\frac{\abs{\mu}-1}{2}$, we have
$\abs\mu-l\geq \frac{\abs{\mu}+1}{2}.$ Then it follows that
\begin{eqnarray*}
  \sum\limits_{l=1}^{
  (\abs\mu-1)/2}\frac{3\abs\mu}{l^{2}(\abs\mu-l)}\leq \frac{6\abs\mu}{\abs\mu+1}\sum\limits_{l=1}^{
  (\abs\mu-1)/2}\frac{1}{l^{2}}<12.
\end{eqnarray*}
In the case when $l\geq \frac{\abs{\mu}+1}{2},$ then we have
\begin{eqnarray*}
  \sum\limits_{l=(\abs\mu+1)/2}^{
  \abs\mu-1}\frac{3\abs\mu}{l^{2}(\abs\mu-l)}\leq \frac{12}{\abs\mu+1}\sum\limits_{l=(\abs\mu+1)/2}^{
  \abs\mu-1}\frac{1}{\abs\mu-l}<12.
\end{eqnarray*}
Combination of the above three inequalities gives the desired
inequality \reff{sum1}. Similarly we can deduce the inequality
\reff{sum2}. Lemma 2.1 is proved.
\end{proof}

Next, we give a following crucial lemma, which is important in the
proof of the main result. Throughout the paper we always assume
$(-i)!=1$ for nonnegative integer $i$.

\begin{lem}\label{crucial}
There exist constants $B$, $C_1$ and $C_2>0$  with $B$ depending
only on the dimension  and $C_1$, $C_2$ depending only on $M_0$,
$E_0$, $H_0$ and $\gamma,$  such that for all multi-indices
$\mu\in\Natural^3$ with $\abs\mu\geq2$ and all $t>0,$ we have
\begin{eqnarray*}
 &&\frac{d}{dt}\norm{\p^\mu f(t)}_{L^2}^2+C_1\norm{\nabla_v\p^\mu  f(t) }_{L^2_\gamma}^2\
  \leq C_2\abs\mu^2\norm{\nabla_v\p^{\mu-1} f(t)}_{L^2_\gamma}^2+{}\\
  &{}&~+C_2 \sum\limits_{2\leq\abs\beta\leq\abs\mu}
  C_{\mu }^\beta \norm{\nabla_v\p^{\mu -\beta+1} f(t)}_{L^2_\gamma}
  \cdot\norm{\nabla_v\p^{\mu-1}
  f(t)}_{L^2_\gamma}\cdot\com{G(f(t))}_{\beta-2}+{}\\
  &{}&~+C_2\sum\limits_{0\leq\abs\beta\leq\abs\mu}
   C_{\mu }^\beta\norm{
  \p^\beta f(t)}_{L^2_\gamma}\cdot\norm{\nabla_v\p^{\mu-1} f(t)}_{L^2_\gamma}
  \cdot\com{G(f(t))}_{{\mu}-\beta},
\end{eqnarray*}
where $C_\mu^\beta=\frac{\mu!}{(\mu-\beta)!\beta!}$ is the binomial
coefficients,
$\com{G(f(t))}_{\omega}=\norm{\p^{\omega}f(t)}_{L^2}+B^{\abs\omega}\inner{\abs\omega-3}!,$
 and $\mu-l$ denotes some multi-index $\tilde\mu$ satisfying
$\tilde\mu\leq\mu$ and $\abs{\tilde\mu}=\abs\mu-l.$
\end{lem}

The proof of Lemma \ref{crucial} will be given in the section
\ref{section3}.

\begin{prop}\label{prp}

Let $f_0$ be the initial datum with finite mass, energy and entropy
and $f(t,v)$  be any solution of the Cauchy problem \reff{Landau'}
satisfying the condition \reff{uniform}. Then for any $T_0,T_1$ with
$0<T_0<T_1<+\infty,$ there exists a constant $A$, depending only on
$M_0$, $E_0$, $H_0$,$\gamma$ and $T_0$,  such that for any $\rho$
with $0<\rho<\min\set{{1\over4},(T_1-T_0)/2},$ and any nonnegative
integer $m,$ the following estimate
\begin{equation}\label{goal}
 \sup_{s\in\Omega_\rho}\norm{\p^\alpha f(s)}_{L^2}
 +\set{\int_{\Omega_\rho}\|\nabla_v
 \p^{\tilde\alpha}f(t)\|_{L^2_{\gamma}}^2\,dt}^{{1\over2}}\leq
 \frac{C_{T_0,T_1}\:A^{m+1}}{\rho^m}\com{(m-2)!}
\end{equation}
holds for any  multi-indices $\alpha$, $\tilde\alpha$ with
$\abs\alpha=\abs{\tilde\alpha}=m,$ where the interval $\Omega_\rho$
and the constant $C_{T_0,T_1}$ are given by
\[
  \Omega_\rho=[T_0+\rho,T_1-\rho],\qquad C_{T_0,T_1}=2(T_1-T_0+1).
\]
\end{prop}

\begin{rmk}
  Note that the constant $A$ in \reff{goal} is independent of
  $T_1,$ which can be deduced from the fact that all $H_\gamma^m$
  norms of $f$ are bounded uniformly in time (see the assumption
  \reff{uniform} above). In fact, the constant $A$ can be calculated explicitly.
   This can be seen in the process of following proof.
\end{rmk}

\begin{proof}[Proof of Proposition \ref{prp}]
For any $\rho$ with $0<\rho<\min\set{{1\over4},(T_1-T_0)/2},$ we
define $\Omega_{\rho,j}, j\geq 1,$ by setting
\[
  \Omega_{\rho,j}=[t_j+\rho, t_j+1-\rho]
\]
with $t_j=T_0+\frac{j-1}{2}.$ Observe $\rho<{1\over4},$ then we can
find a positive integer $N_0$ with $N_0\leq 2(T_1-T_0+1),$ such that
\[
  \Omega_\rho=[T_0+\rho, T_1-\rho]\subset
  \bigcup\limits_{j=1}^{N_0}\Omega_{\rho,j}.
\]
Hence the desired estimate \reff{goal} will follow if we can find a
constant $A$, depending only on $T_0,$ $M_0$, $E_0$, $H_0$ and
$\gamma$, such that for any $\rho$ with
$0<\rho<\min\set{{1\over4},(T_1-T_0)/2},$ and any nonnegative
integer $m,$ the following estimate
\begin{equation}\label{+goal}
 \sup_{s\in\Omega_{\rho,j}}\norm{\p^\alpha f(s)}_{L^2}
 +\set{\int_{\Omega_{\rho,j}}\|\nabla_v \p^{\tilde\alpha}f(t)\|_{L^2_{\gamma}}^2\,dt}^{{1\over2}}\leq
 \frac{A^{m+1}}{\rho^m}\com{(m-2)!}
\end{equation}
holds for all  multi-indices $\alpha$, $\tilde\alpha$ with
$\abs\alpha=\abs{\tilde\alpha}=m,$ and  all $1\leq j\leq N_0.$

We shall use induction on $m$ to prove the estimate \reff{+goal}.
First, we take a constant $A$ large enough such that
\begin{equation}\label{ac}
 A\geq 2\sup_{s\geq T_0}\norm{f(s)}_{H^2_\gamma}+2.
\end{equation}
In view of \reff{uniform}, we see that the constant $A$ depends only
on $T_0,$ $M_0$, $E_0$, $H_0$ and $\gamma.$ Observing that
$|\Omega_{\rho,j}|,$ the Lebesgue measure of $\Omega_{\rho,j},$ is
less than 1 and that
\[
  \sup_{s\in\Omega_{\rho,j}}\norm{f(s)}_{H^2_\gamma}\leq
  \sup_{s\geq t_j}\norm{f(s)}_{H^2_\gamma}\leq \sup_{s\geq
  T_0}\norm{f(s)}_{H^2_\gamma}\leq {A\over2}, \quad 1\leq j\leq N_0,
\]
we compute, for any $\gamma, \tilde\gamma$ with
$\abs\gamma=\abs{\tilde\gamma}\leq 1, $
\begin{eqnarray*}
  \sup_{s\in\Omega_{\rho,j}}\norm{\p^\gamma f(s)}_{L^2}
  +\set{\int_{\Omega_{\rho,j}}\|\nabla_v \p^{\tilde\gamma} f(t)\|_{L^2_{\gamma}}^2\,dt}^{{1\over2}}
  \leq {A\over2}+{A\over2}\leq A, \quad 1\leq j\leq N_0,
\end{eqnarray*}
which implies that the estimate \reff{+goal} holds for $m=0, 1$.

We assume, for some integer $k\geq2,$ the estimate \reff{+goal}
holds for all $m$ with $m\leq k-1$. We now need to prove the
validity of \reff{+goal} for $m=k,$ or equivalently, to show the
following two estimates: for any
$0<\rho<\min\set{{1\over4},~(T_1-T_2)/2},$
\begin{equation}\label{first term}
\sup_{s\in\Omega_{\rho,j}}\norm{\p^\alpha f(s)}_{L^2}\leq {1\over
2}\frac{A^{\abs{\alpha}+1}}{\rho^{\abs{\alpha}}}\com{(\abs\alpha-2)!},\quad
\forall~\abs\alpha=k,
\end{equation}
and
\begin{equation}\label{second term}
 \set{\int_{\Omega_{\rho,j}}\|\nabla_v
 \p^{\tilde\alpha}f(t)\|_{L^2_{\gamma}}^2dt}^{1\over2}\leq
 {1\over2}\frac{A^{\abs{\tilde \alpha}+1}}{\rho^{\abs{\tilde\alpha}}}\com{(\abs{\tilde\alpha}-2)!},\quad
\forall~\abs{\tilde\alpha}=k.
\end{equation}
In the following discussion, we fix $j, \rho, \alpha$ and
$\tilde\alpha,$ with $1\leq j\leq N_0,
0<\rho<\min\set{{1\over4},(T_1-T_0)/2}$ and
$\abs\alpha=\abs{\tilde\alpha}=k$. In this case we introduce a
cut-off function $\vpi(t)$ which is a smooth function with compact
support in $\Omega_{\tilde\rho,j}$, where $
\tilde\rho=\frac{k}{k+1}\rho,$ and satisfies $0\leq\vpi\leq1$ and
$\varphi=1$ in $\Omega_{\rho,j}.$ It is easy to see that
  \begin{equation}\label{vpi}
  \sup\limits_{t\in\Real}\abs{\frac{d\varphi(t)}{dt}}\leq
      \bar C\:k/\rho,
  \end{equation}
where the constant $\bar C$ is independent of $k$ and $\rho,$ and
\begin{equation}\label{fact}
 {1\over{\rho}^{k}}\leq{1\over{\tilde\rho}^{k}}={1\over{\rho}^{k}}\times\big({{k+1}\over{k}}\big)^{k}
\leq  {3\over{\rho}^{k}}.
\end{equation}

First we prove the estimate \reff{first term}. By using Lemma
\ref{crucial}, one has
\begin{eqnarray*}
  &&\frac{d}{dt}\norm{\p^\alpha f(t)}_{L^2}^2+C_1\norm{\nabla_v\p^\alpha  f(t) }_{L^2_\gamma}^2
  \leq C_2\abs\alpha^2\norm{\nabla_v\p^{\alpha-1} f(t)}_{L^2_\gamma}^2+{}\\
  &{}&~~+C_2 \sum\limits_{2\leq\abs\beta\leq\abs\alpha}
  C_{\alpha }^\beta \norm{\nabla_v\p^{\alpha -\beta+1} f(t)}_{L^2_\gamma}
   \norm{\nabla_v\p^{\alpha-1}
  f(t)}_{L^2_\gamma}\cdot\com{G(f(t))}_{\beta-2}+{}\\
  &{}&~~+C_2\sum\limits_{0\leq\abs\beta\leq\abs\alpha}
   C_{\alpha }^\beta\norm{
  \p^\beta f(t)}_{L^2_\gamma}\cdot\norm{\nabla_v\p^{\alpha-1} f(t)}_{L^2_\gamma}
  \cdot\com{G(f(t))}_{{\alpha}-\beta}.
\end{eqnarray*}
Rewriting  the last term of the right hand side as
\begin{eqnarray*}
  &&C_2\norm{
  f(t)}_{L^2_\gamma}\cdot\norm{\nabla_v\p^{\alpha-1} f(t)}_{L^2_\gamma}
  \cdot\com{G(f(t))}_{{\alpha}}\\
  &&+C_2\sum\limits_{1\leq\abs\beta\leq\abs\alpha}
   C_{\alpha }^\beta\norm{
  \p^\beta f(t)}_{L^2_\gamma}\cdot\norm{\nabla_v\p^{\alpha-1} f(t)}_{L^2_\gamma}
  \cdot\com{G(f(t))}_{{\alpha}-\beta},
\end{eqnarray*}
we obtain that
\begin{eqnarray*}
  &&\frac{d}{dt}\norm{\p^\alpha f(t)}_{L^2}^2+C_1\norm{\nabla_v\p^\alpha  f(t) }_{L^2_\gamma}^2
  \leq C_2\abs\alpha^2\norm{\nabla_v\p^{\alpha-1} f(t)}_{L^2_\gamma}^2+{}\\
  &{}&~~+C_2\norm{
  f(t)}_{L^2_\gamma}\cdot\norm{\nabla_v\p^{\alpha-1} f(t)}_{L^2_\gamma}
  \cdot\com{G(f(t))}_{{\alpha}}+{}\\
  &{}&~~+C_2 \sum\limits_{2\leq\abs\beta\leq\abs\alpha}
  C_{\alpha }^\beta \norm{\nabla_v\p^{\alpha -\beta+1} f(t)}_{L^2_\gamma}
  \norm{\nabla_v\p^{\alpha-1}
  f(t)}_{L^2_\gamma}\cdot\com{G(f(t))}_{\beta-2}+{}\\
  &{}&~~+C_2\sum\limits_{1\leq\abs\beta\leq\abs\alpha}
   C_{\alpha }^\beta\norm{
  \p^\beta f(t)}_{L^2_\gamma}\cdot\norm{\nabla_v\p^{\alpha-1} f(t)}_{L^2_\gamma}
  \cdot\com{G(f(t))}_{{\alpha}-\beta}.
\end{eqnarray*}
Multiplying by the cut-off function $\vpi(t)$ in the both sides of
the inequality above to get
\begin{eqnarray*}
  &&\frac{d}{dt}\inner{\varphi(t)\norm{\p^\alpha f(t)}_{L^2}^2}+C_1\varphi(t)\norm{\nabla_v\p^\alpha  f(t)
  }_{L^2_\gamma}^2\\
  &&{}\leq \frac{d\varphi}{dt}\cdot\norm{\p^\alpha f(t)}_{L^2}^2+C_2\cdot\varphi(t)\abs\alpha^2
  \norm{\nabla_v\p^{\alpha-1} f(t)}_{L^2_\gamma}^2+{}\\
  &{}&+C_2\cdot\varphi(t)~\norm{f(t)}_{L^2_\gamma}\cdot\norm{\nabla_v\p^{\alpha-1}f(t)}_{L^2_\gamma}
  \cdot\com{G(f(t))}_{{\alpha}}+{}\\
  &{}&+C_2 \varphi(t)\sum\limits_{2\leq\abs\beta\leq\abs\alpha}
  C_{\alpha }^\beta \norm{\nabla_v\p^{\alpha -\beta+1} f(t)}_{L^2_\gamma}
  \norm{\nabla_v\p^{\alpha-1}
  f(t)}_{L^2_\gamma}\com{G(f(t))}_{\beta-2}+{}\\
  &{}&+C_2\cdot\varphi(t)\sum\limits_{1\leq\abs\beta\leq\abs\alpha}
   C_{\alpha }^\beta\norm{
  \p^\beta f(t)}_{L^2_\gamma}\cdot\norm{\nabla_v\p^{\alpha-1} f(t)}_{L^2_\gamma}
  \cdot\com{G(f(t))}_{{\alpha}-\beta}.
\end{eqnarray*}
In the sequel,  we set
\begin{equation}\label{++G}
\com{G(f)}_{\rho,\beta}=\sup_{t\in\Omega_{\rho,j}}\com{G(f(t))}_{\beta}
=\sup_{t\in\Omega_{\rho,j}}\norm{\p^\beta
f(t)}_{L^2}+B^{\abs\beta}(\abs\beta-3)!.
\end{equation}
Since supp $\varphi\subset\Omega_{\tilde\rho,j}$ with
$\tilde\rho={{k\rho}\over {k+1}},$ and $\varphi(t)=1$ for all
$t\in\Omega_{\rho,j}$ and $\varphi(t_j)=0$, then for any
$s\in\Omega_{\rho,j},$  we integrate the inequality above over the
interval $[t_j, s]\subset[t_j,t_j+1-\rho]$ to get, from Cauchy
inequality, that
\begin{eqnarray*}
  \norm{\p^\alpha f(s)}_{L^2}^2&=&\varphi(s)\norm{\p^\alpha f(s)}_{L^2}^2-\varphi(t_j)~\norm{\p^\alpha
  f(t_j)}_{L^2}^2 \\
  &\leq &(S_1)+(S_2)+(S_3)+(S_4)+(S_5),
\end{eqnarray*}
where $(S_j)$, $1\leq j\leq 5,$ are given by
\[
(S_1)=\sup\limits_{t\in\Real}\abs{\frac{d\varphi}{dt}}\int_{\Omega_{\tilde\rho,j}}\norm{\p^{\alpha}
   f(t)}_{L^2}^2dt;
\]
\[
  (S_2)=C_2\abs\alpha^2\int_{\Omega_{\tilde\rho,j}}\norm{\nabla_v\p^{\alpha-1}
  f(t)}_{L^2_\gamma}^2dt;
\]
\[
  (S_3)=C_2
  \sup_{t\in\Omega_{\tilde\rho,j}}\norm{f(t)}_{L^2_\gamma}\cdot
  \set{\int_{\Omega_{\tilde\rho,j}}\com{G(f(t))}_{\alpha}^2dt}^{1\over2}
  \set{\int_{\Omega_{\tilde\rho,j}}
  \norm{\nabla_v\p^{\alpha-1}f(t)}_{L^2_\gamma}^2dt}^{1\over2};
\]
\begin{eqnarray*}
    (S_4)&=&C_2\sum\limits_{2\leq\abs\beta\leq\abs\alpha}
    C_{\alpha}^\beta
    \com{G(f)}_{\tilde\rho,~\beta-2}\set{\int_{\Omega_{\tilde\rho,j}}\norm{\nabla_v\p^{\alpha-\beta+1}f(t)}
    _{L^2_\gamma}^2dt}^{{1\over2}}\times{}\\
    &{}&\times\set{\int_{\Omega_{\tilde\rho,j}}\norm{\nabla_v\p^{\alpha-1}f(t)}_{L^2_\gamma}
    ^2dt}^{{1\over2}};
\end{eqnarray*}
and
\begin{eqnarray*}
  (S_5)&=&C_2\sum\limits_{1\leq\abs\beta\leq\abs\alpha}
  C_{\alpha}^\beta
  \com{G(f)}_{\tilde\rho, \alpha-\beta}\set{\int_{\Omega_{\tilde\rho,j}}\norm{\p^{\beta}f(t)}_{L^2_\gamma}^2dt}^{1\over2}
  \times{}\\
  &{}&\times\set{\int_{\Omega_{\tilde\rho,j}}\norm{\nabla_v\p^{\alpha-1}f(t)}_{L^2_\gamma}^2dt}^{1\over2}.
\end{eqnarray*}

In order to treat these terms, we need the following estimates which
can be deduced directly from the the induction hypothesis. The
validity of  \reff{+goal} for all $m\leq k-1$ implies that
\begin{equation}\label{ind1}
\set{\int_{\Omega_{\tilde\rho,j}}\norm{\nabla_v\p^{\gamma}f(t)}_{L^2_\gamma}^2dt}^{{1\over2}}
 \leq \frac{A^{\abs\gamma+1}}{\tilde\rho^{\abs\gamma}}\com{(\abs\gamma-2)!},\quad
 0\leq\abs\gamma\leq k-1 ;
\end{equation}
\begin{equation}\label{+ab}
\sup_{t\in\Omega_{\tilde\rho,j}}\norm{\p^{\lambda}f(t)}_{L^2}
 \leq \frac{A^{\abs\lambda+1}}{\tilde\rho^{\abs\lambda}}\inner{(\abs\lambda-2)!},\quad
 0\leq \abs\lambda\leq k-1;
\end{equation}
and
\begin{equation}\label{ind3}
\set{\int_{\Omega_{\tilde\rho,j}}\norm{\p^{\beta}f(t)}_{L^2_\gamma}^2dt}^{1/2}
 \leq
\frac{ A^{\abs\beta}}{
\tilde\rho^{\abs\beta-1}}\com{(\abs\beta-3)!},\quad
 1\leq\abs\beta\leq\abs\alpha,
\end{equation}
the last inequality following from the fact $\norm {\p^\beta
f}_{L^2_\gamma}\leq \norm {\nabla_v\p^{\beta-1} f}_{L^2_\gamma} $
for any multi-index $\beta$ with $1\leq\abs\beta\leq\abs\alpha.$
Consequently, if we take $A$ large enough such that $A\geq B,$ then
in view of \reff{++G}, \reff{+ab} and \reff{ind3}, one has
\begin{equation}\label{mfv1}
 \com{G(f)}_{\tilde\rho,\lambda}
 \leq \frac{2A^{\abs\lambda+1}}{\tilde\rho^{\abs\lambda}}\inner{(\abs\lambda-2)!},
 \quad 0\leq \abs\lambda\leq\abs\alpha-1,
\end{equation}
and
\begin{equation}\label{mfv3}
 \set{\int_{\Omega_{\tilde\rho,j}}\com{G(f(t))}_{\alpha}^2dt}^{1/2}
 \leq\frac{2A^{|\alpha|}}{\tilde\rho^{|\alpha|-1}}\com{(|\alpha|-3)!}.
\end{equation}

Now we are ready to handle the terms $(S_j)$ for $1\leq j\leq 5.$ In
the process below, the notations $C_\ell$, $\ell\geq 3$, will be
used to denote different constants which are larger than 1 and
depend only on $M_0, E_0, H_0, \gamma$ and $T_0.$ Observe
$\norm{\p^\alpha f(s)}_{L^2}\leq \norm{\p^\alpha
f(s)}_{L^2_\gamma},$ one has, by \reff{vpi} and \reff{ind3},
 \begin{equation}\label{S1}
 (S_1)\leq
 {C_3k\over\rho}\set{\frac{A^{\abs\alpha}}{\tilde\rho^{\abs\alpha-1}}\com{(\abs\alpha-3)!}}^2
 \leq
 C_4\set{\frac{A^{\abs\alpha}}{\tilde\rho^{\abs\alpha}}\com{(\abs\alpha-2)!}}^2,
\end{equation}
where we used the fact that
$\rho^{-1}<\tilde\rho^{-1}<\tilde\rho^{-2}$ and that $\abs\alpha=k.$
Next, by virtue of \reff{ind1}, we obtain
 \begin{equation}\label{S2}
 (S_2)\leq
  C_2\abs\alpha^2\set{\frac{A^{\abs\alpha}}{\tilde\rho^{\abs\alpha-1}}\com{(\abs\alpha-3)!}}^2
  \leq C_5\set{\frac{A^{\abs\alpha}}{\tilde\rho^{\abs\alpha-1}}\com{(\abs\alpha-2)!}}^2 .
\end{equation}

To estimate the term $(S_3),$ we use the estimates \reff{ac},
\reff{ind1} and \reff{mfv3}, which gives
\begin{equation}\label{S4}
  (S_3)\leq
   C_6 A
  \set{\frac{A^{\abs\alpha}}{\tilde\rho^{\abs\alpha-1}}\com{(\abs\alpha-3)!}}^2.
\end{equation}

Now we handle the term $(S_4).$ By virtue of \reff{ind1} and
\reff{mfv1}, we can deduce that  the term $(S_4)$ is bounded by
\begin{eqnarray*}
 &&\sum\limits_{2\leq\abs\beta\leq\abs\alpha}
 \frac{C_2 \abs\alpha
 !}{\abs\beta!(\abs{\alpha}-\abs\beta)!}\frac{A^{|\beta|-1}}{\tilde\rho^{|\beta|-2}}\com{(|\beta|-4)!}
 \frac{A^{\abs{\alpha}-|\beta|+2}}{\tilde\rho^{\abs{\alpha}-|\beta|+1}}\com{(\abs\alpha-|\beta|-1)!}\times{}\\
 &{}&\indent\times \frac{A^{|\alpha|}}{\tilde\rho^{|\alpha|-1}}\com{(|\alpha|-3)!},
\end{eqnarray*}
that is, from the estimate \reff{sum1} of Lemma \ref{lemma1}, we
have
\begin{eqnarray}\label{S3}
   (S_4)&\leq&
  C_{7}A\set{\frac{A^{\abs{\alpha}}}{\tilde\rho^{\abs{\alpha}-1}}\com{(\abs\alpha-2)!}}^2\times
  \set{\sum\limits_{2\leq\abs\beta\leq\abs\alpha-1}
 \frac{\abs\alpha}{\abs\beta^4(\abs{\alpha}-\abs\beta)}+1}\nonumber\\
 &\leq &C_{8}  A\set{
  \frac{A^{\abs{\alpha}}}{\tilde\rho^{\abs{\alpha}-1}}\com{(\abs\alpha-2)!}}^2.
\end{eqnarray}
Similarly, by virtue of \reff{sum2}, \reff{ind1}, \reff{ind3} and
\reff{mfv1}, we can get the estimate for the term $(S_5)$, i.e.
\begin{equation}\label{S5}
 (S_5)\leq C_{9} A \set{
  \frac{A^{\abs{\alpha}}}{\tilde\rho^{\abs{\alpha}-1}}\com{(\abs\alpha-2)!}}^2.
\end{equation}
Combining the estimates \reff{S1}-\reff{S5}, one has, for any
$s\in\Omega_{\rho,j}$, that
\begin{eqnarray}\label{15}
  \norm{\p^\alpha
  f(s)}_{L^2}^2&\leq& \sum\limits_{i=1}^5(S_i)
  \leq C_{10}A
  \set{\frac{A^{|\alpha|}}{\tilde\rho^{|\alpha|}}\com{(|\alpha|-2)!}}^2\nonumber\\
  &\leq & C_{11}A
  \set{\frac{A^{|\alpha|}}{\rho^{|\alpha|}}\com{(|\alpha|-2)!}}^2,
\end{eqnarray}
the last inequality above follows from the estimate \reff{fact}.
Taking $A$ large enough such that
\[
A\geq 4\max\set{\sup\limits_{s\geq T_0}\norm {f(s)}_{H^2_\gamma}+1,
\: B,\: C_{11}},
\]
then we obtain finally
\[\norm{\p^\alpha
  f(s)}_{L^2}^2\leq \set{{1\over2}\frac{A^{\abs\alpha+1}}{\rho^{\abs\alpha}}\com{(\abs\alpha-2)!}}^{2},
\quad  \forall\:s\in\Omega_{\rho,j},
\]
which gives the proof of the estimate \reff{first term}.

Now, it remains to prove the estimate \reff{second term}, which can
be handled similarly as the proof of the estimate \reff{first term}.
Let us apply Lemma \ref{crucial} again with $\mu=\tilde\alpha$, and
then we multiply the cut-off function $\varphi(t)$ in the both sides
of the estimate in Lemma \ref{crucial} to get
\begin{eqnarray*}
  &&\frac{d}{dt}\inner{\varphi(t)\norm{\p^{\tilde\alpha} f(t)}_{L^2}^2}+C_1\varphi(t)\norm{\nabla_v\p^{\tilde\alpha}  f(t)
  }_{L^2_\gamma}^2\\
  &&{}\leq \frac{d\varphi}{dt}\cdot\norm{\p^{\tilde\alpha} f(t)}_{L^2}^2+C_2\cdot\varphi(t)\abs{\tilde\alpha}^2
  \norm{\nabla_v\p^{{\tilde\alpha}-1} f(t)}_{L^2_\gamma}^2+{}\\
  &{}&+C_2\cdot\varphi(t)~\norm{f(t)}_{L^2_\gamma}\cdot\norm{\nabla_v\p^{{\tilde\alpha}-1}f(t)}_{L^2_\gamma}
  \cdot\com{G(f(t))}_{{{\tilde\alpha}}}+{}\\
  &{}&+C_2 \varphi(t)\sum\limits_{2\leq\abs\beta\leq\abs{\tilde\alpha}}
  C_{{\tilde\alpha} }^\beta \norm{\nabla_v\p^{{\tilde\alpha} -\beta+1} f(t)}_{L^2_\gamma}
  \norm{\nabla_v\p^{{\tilde\alpha}-1}
  f(t)}_{L^2_\gamma}\com{G(f(t))}_{\beta-2}+{}\\
  &{}&+C_2\cdot\varphi(t)\sum\limits_{1\leq\abs\beta\leq\abs{\tilde\alpha}}
   C_{{\tilde\alpha} }^\beta\norm{
  \p^\beta f(t)}_{L^2_\gamma}\cdot\norm{\nabla_v\p^{{\tilde\alpha}-1} f(t)}_{L^2_\gamma}
  \cdot\com{G(f(t))}_{{{\tilde\alpha}}-\beta}\\
  &&\stackrel{{\rm def}}{=}\mathcal {N}(t).
\end{eqnarray*}
Observe supp $\varphi\subset \Omega_{\tilde\rho,j},$ then
integrating the above inequality over the interval
$\Omega_{\tilde\rho,j}$ yields that
\[
  C_1\int_{\Omega_{\tilde\rho,j}}\varphi(s)\norm{\nabla_v\p^{\tilde\alpha}f(s)}_{L^2_{\gamma}}^2ds
  \leq \int_{\Omega_{\tilde\rho,j}}\mathcal {N}(t)dt.
\]
Repeating the previous arguments we used to estimate the terms
$(S_1)$-$(S_5)$, one has
\[
   \int_{\Omega_{\tilde\rho,j}}\mathcal {N}(t)dt \leq C_{11}A
  \set{\frac{A^{|\alpha|}}{\rho^{|\alpha|}}\com{(|\alpha|-2)!}}^2,
\]
where $C_{11}$ is the same constant as appeared in \reff{15}. From
these inequalities, together with the fact that $A\geq 4C_{11}$ and
\[
\int_{\Omega_{\rho,j}}\norm{\nabla_v\p^{\tilde\alpha}f(s)}_{L^2_{\gamma}}^2ds\leq
 \int_{\Omega_{\tilde\rho,j}}\varphi(s)\norm{\nabla_v\p^{\tilde\alpha}f(s)}_{L^2_{\gamma}}^2ds,
\]
we can deduce that the estimate \reff{second term} holds. The proof
of Proposition \ref{prp} is completed.

\end{proof}

Now we present the proof of the main result.

\begin{proof}[Proof of Theorem \ref{main result}]
Given $t_0>0$, and for any $t\geq t_0,$ we take
$T_0=\frac{3t_0}{4}$, $T_1=t+\frac{3t_0}{4}$ and
$\rho=\frac{t_0}{4}$ in the estimate \reff{goal}. This gives, for
any $\abs\alpha=m\geq 0,$
  \[
  \norm{\p^\alpha f(t)}_{L^2}\leq
  \sup_{s\in[T_0+\rho,T_1-\rho]}\norm{\p^\alpha f(s)}_{L^2}\leq
  \frac{2(t+1)(4A)^{m+1}}{t_0^m}\com{(m-2)!},
  \]
  with the constant $A$ depending only on $M_0, E_0, H_0, \gamma$ and
  $t_0$. Thus one has
  \[
  \frac{c_0^m}{m!}\sum\limits_{\abs\alpha=m}\norm{\p^\alpha f(t)}_{L^2}\leq
  16A(t+1)\inner{1\over2}^m,
  \]
where  $c_0=\frac{t_0}{16A}.$ Hence we have
  \[\norm{e^{c_0\inner{-\triangle_v}^{1\over2}}f(t,\cdot)}_{L^2(\Real^3)}\leq \tilde A(t+1),\]
  where the constant $\tilde A$ depends only on $M_0, E_0, H_0, \gamma$ and $t_0.$
  Theorem \ref{main result} is proved.
\end{proof}

\section{Proof of Lemma \ref{crucial}}
\label{section3}

This section is devoted to the proof of Lemma \ref{crucial}. Our
starting point is the following uniformly ellipticity property of
the matrix $(\bar a_{ij}),$ cf. Proposition 4 in \cite{desv-vill1}.

\begin{lem}
 There exists a constant $K$, depending only on $\gamma$ and $M_0, E_0, H_0$, such
 that
 \begin{equation}\label{ellipticity}
 \sum_{1\leq i,j\leq 3} \bar a_{ij}(t,v)\xi_i\xi_j\geq K(1+\abs
  v^2)^{\gamma/2}\abs\xi^2, \quad \forall~\xi\in\Real^3.
 \end{equation}
\end{lem}

\begin{rmk}
Although the ellipticity of $(\bar a_{ij})$ was proved in
\cite{desv-vill1} in the hard potential case $\gamma>0,$   it's
still true for the Maxwellian case $\gamma=0$. This can be seen in
the proof of Proposition 4 in \cite{desv-vill1}.
\end{rmk}

\begin{lem}\label{analyticfunction}
There exists a constant $L,$ depending only on the space dimension,
such that for any fixed positive integer $N$, $N\geq2,$ one can find
a function $\psi_N\in C_0^\infty(\Real^3)$ with compact support in
$\set{v\in\Real^3\V \abs v\leq 2},$ satisfying that
$0\leq\psi_N(v)\leq 1$ and $\psi_N(v)=1$ on the ball
$\set{v\in\Real^3\V \abs v\leq 1}$, and
\begin{equation}\label{psi}
   \sup \abs {\p^{\lambda} \psi_N}\leq
   (LN)^{|\lambda|},\quad \forall~\lambda, \:\abs\lambda\leq N.
\end{equation}

\end{lem}

\begin{proof}
For the construction of $\psi_N$,  we refer to \cite{Hormander,Ro}.
Choose a non-negative function $\rho\in C_0^\infty(\Real^3)$ with
compact support in $\{v\in\Real^3\V \abs v\leq {1\over2}\},$ and
$\int_{\Real^3}\rho(v)dv=1.$ Set
\[L=\max_{\abs{\alpha}\leq 1} \int_{\Real^3}\abs{\p^\alpha \rho(v)}dv.\]
Let $\chi$ be the characteristic function of the set
$\set{v\in\Real^3\V \abs v\leq {3\over2}}.$ For each $r\geq0,$
defined
\[\rho_r(v)=r^{-3}\rho(v/r).\]
Set then
\[\psi_N=\chi*\rho_{1/N}*\cdots*\rho_{1/N}\]
with $N$ factors $\rho_{1/N}$. Direct verification shows that the
function $\psi_N$ satisfies the desired properties.
\end{proof}

\begin{lem}\label{a}

There exists a constant $B$, depending only on the dimension,  such
that  for all multi-indices $\beta$ with $|\beta|\geq2$ and all
$g,h\in L^2_\gamma(\Real^3),$ one has
\begin{eqnarray*}
&\sum_{1\leq i,j\leq 3}\int_{\Real^3}(\p^{\beta}\bar
a_{ij}(t,v))g(v)h(v)dv \leq
 C\norm{g}_{L^2_{\gamma}}\norm{h}_{L^2_{\gamma}}\com{G(f(t))}_{\beta-2},
 \quad \forall\:t>0,
\end{eqnarray*}
where $\com{G(f(t))}_{\beta-2}=\set{
 \norm{\p^{\beta-2}f(t)}_{L^2}+B^{\abs{\beta}-2}(|\beta|-5)!}.$
\end{lem}

\begin{proof}
Let $L$ be the constant given in Lemman \ref{analyticfunction}, and
let $\psi=\psi_{\abs\beta}\in C_0^\infty(\Real^3)$ be the function
constructed in Lemma \ref{analyticfunction} for $N=\abs\beta.$
 Write
$a_{ij}=\psi a_{ij}+(1-\psi) a_{ij}.$ Then $\bar a_{ij}=(\psi
a_{ij})*f+[(1-\psi) a_{ij}]*f$, and hence
\[\p^{\beta}\bar
a_{ij}=\com{\p^{\tilde\beta}(\psi
a_{ij})}*(\p^{\beta-\tilde\beta}f)+\set{\p^{\beta}\com{(1-\psi)
a_{ij}}}*f,
\]
where $\tilde\beta$ is an arbitrary multi-index satisfying
$\tilde\beta\leq\beta$ and $|\tilde\beta|=2.$

We first treat the term $\com{\p^{\tilde\beta}(\psi
a_{ij})}*(\p^{\beta-\tilde\beta}f)$. It is easy to verify that for
all $\tilde\beta$ with $\bigabs{\tilde\beta}=2,$
\[\bigabs{(\p^{\tilde\beta}a_{ij})(v-\vstar)}\leq C\abs{v-\vstar}^\gamma,\]
thus we can compute
\begin{eqnarray*}
  \abs{\com{\p^{\tilde\beta}(\psi
  a_{ij})}*(\p^{\beta-\tilde\beta}f)(v)}&=&\abs{\int_{\Real^3}
  \com{\p^{\tilde\beta}(\psi
  a_{ij})}(v-\vstar)\cdot(\p^{\beta-\tilde\beta}f)(\vstar)d\vstar}\\
  &\leq &C\int_{\set{\abs{\vstar-v}\leq2}}
  \abs{(\p^{\beta-\tilde\beta}f)(\vstar)}d\vstar\\
  &\leq &C\norm{\p^{\beta-2}f(t)}_{L^2}.
\end{eqnarray*}
Here the notation $C$ is used to denote different constants which
will depend only on the $\gamma$, $M_0$, $E_0$ and $H_0$.

In the next step we treat the term $\set{\p^{\beta}\com{(1-\psi)
a_{ij}}}*f$. By using Leibniz's formula, one has
\begin{eqnarray*}
  &&\abs{\biginner{\p^{\beta}\com{(1-\psi) a_{ij}}}*f(v)}\\&&=
  \big|\sum\limits_{0\leq\abs{\lambda}\leq\abs{\beta}}C_\beta^\lambda\int_{\Real^3}
  \com{\p^{\beta-\lambda}(1-\psi)}(v-\vstar)\cdot\inner{\p^{\lambda}
  a_{ij}}(v-\vstar)\cdot f(t,\vstar)d\vstar\big|\\
  &&\leq J_1+J_2,
\end{eqnarray*}
where
\[
  J_1=\sum\limits_{0\leq\abs{\lambda}\leq\abs{\beta}-1}C_\beta^\lambda
  \int_{\set{1\leq\abs{\vstar-v}\leq2}}
  \abs{\p^{\beta-\lambda}\psi(v-\vstar)}\cdot\abs{\p^{\lambda}
  a_{ij}(v-\vstar)}f(t,\vstar)d\vstar,
\]
and
\[
J_2=\int_{\set{\abs{\vstar-v}\geq1}}
  \abs{1-\psi(v-\vstar)}\cdot\abs{\p^{\beta}
  a_{ij}(v-\vstar)}\cdot f(t,\vstar)d\vstar .
\]
In view of \reff{coe}, we can find a constant $\tilde C,$ such that
for all multi-indices $\lambda\leq\beta,$
\[
\abs{\inner{\p^{\lambda} a_{ij}}(v-\vstar)}\leq \tilde
C^{\abs\lambda}\abs\lambda!, \quad {\rm
for}~1\leq\abs{\vstar-v}\leq2,
\]
which along with the estimate \reff{psi} gives
\begin{eqnarray*}
 J_1&\leq& {L}^{\abs\beta}(\abs\beta)^{\abs\beta}\cdot\norm{f(t)}_{L^1}
 \sum\limits_{0\leq\abs{\lambda}\leq\abs{\beta}-1}\inner{\frac{\tilde
 C}{L}}^{\abs\lambda}\\
 &\leq& 30^{\abs\beta}{L}^{\abs\beta}(\abs\beta-5)!\norm{f(t)}_{L^1}
 \sum\limits_{0\leq\abs{\lambda}\leq\abs{\beta}-1}\inner{\frac{\tilde
 C}{L}}^{\abs\lambda},
\end{eqnarray*}
the last estimate follows from the fact that
\[
\abs\beta^{\abs\beta}\leq e^{\abs\beta}\abs\beta!\leq
30^{\abs\beta}(\abs\beta-5)!.
\]
Furthermore, it is easy to see that, for all $\beta$ with
$\abs\beta\geq2,$
\[
\abs{\inner{\p^{\beta} a_{ij}}(v-\vstar)}\leq \tilde
C^{\abs\beta}\abs\beta!(1+\abs\vstar^\gamma+\abs v^\gamma),\quad
{\rm for\:\:}\abs{\vstar-v}\geq1.
\]
Hence
 \[ J_2\leq2\tilde
C^{\abs\beta}\abs\beta!\cdot\norm{f(t)}_{L_\gamma^1}(1+\abs
v^\gamma)\leq30^{\abs\beta}\tilde
C^{\abs\beta}(\abs\beta-5)!\cdot\norm{f(t)}_{L_\gamma^1}(1+\abs
v^\gamma).
 \]
So we may let $L\geq 2\tilde C$ and then take $B$ large enough such
that $B\geq 30L$, thus it follows from the estimates above that
\begin{eqnarray*}
  J_1+J_2\leq (50L)^2\norm{f(t)}_{L_\gamma^1}
  B^{\abs\beta-2}(\abs\beta-5)!(1+\abs v^2)^{\gamma/2}.
\end{eqnarray*}
This along with the  fact $\norm{f(t)}_{L_\gamma^1}\leq M_0+2E_0$
gives that
\begin{eqnarray*}
 \abs{\biginner{\p_v^{\beta}\com{(1-\psi) a_{ij}}}*f(v)}\leq J_1+J_2\leq
 C B^{\abs\beta-2}(\abs\beta-5)!(1+\abs v^2)^{\gamma/2}.
\end{eqnarray*}
Combining the estimate on the term $\com{\p^{\tilde\beta}(\psi
a_{ij})}*(\p^{\beta-\tilde\beta}f)$ yields that
\begin{eqnarray*}
  \abs{\p^{\beta}\bar a_{ij}(v)}&\leq& C
  \set{\norm{\p^{{\beta}-2}f(t)}_{L^2}
  +B^{\abs{\beta}-2}\com{(|\beta|-5)!}\cdot(1+\abs v^2)^{\gamma/2}}\\
  &\leq& C \com{G(f(t))}_{{\beta}-2}\cdot
  (1+\abs v^2)^{\gamma/2}.
\end{eqnarray*}
By using the Cauchy's inequality and the estimates above, one can
deduce the desired inequality in Lemma \ref{a}.
\end{proof}

Similarly,  we have following estimate

\begin{lem}\label{c}
For all multi-indices $\beta$ with $|\beta|\geq0$ and all $g,h\in
L^2_\gamma(\Real^3),$  one has
\begin{eqnarray*}
&\int_{\Real^3}(\p^{\beta}\bar c(t,v))g(v)h(v)dv \leq
 C\norm{g}_{L^2_\gamma}\norm{h}_{L^2_\gamma}
 \cdot\com{G(f(t))}_{\beta}, \quad
\forall\: t\geq0.
\end{eqnarray*}
\end{lem}
\bigskip

Next, we give the proof of Lemma \ref{crucial}.

\begin{proof}[Proof of Lemma \ref{crucial}]Set $b_j=\sum_{1\leq i\leq 3}\p_{v_i} a_{ij}(v)=-2\abs v^\gamma v_j$, thus we have
 $\sum_{1\leq i\leq 3}\p_{v_i}\bar a_{ij}(v)=\bar b_j(v)$, $
\sum_{1\leq j\leq 3}\p_{v_j}\bar b_j=\bar c$. Since the solution $f$
satisfies
$$
\p_tf=\sum_{1\leq i,j\leq 3}\bar a_{ij}\p_{v_ivj}f-\bar c f,
$$
thus we have
\begin{eqnarray*}
 \frac{d}{dt}\norm{\p^\mu f(t)}_{L^2}^2&=&2\int_{\Real^3}\com{\p_t
 \p^\mu f(t,v)}\cdot\com{\p^\mu  f(t,v)}dv\\&=&2\sum_{1\leq i,j\leq 3}\int_{\Real^3}\com{
 \p^\mu  \biginner{ \bar a_{ij}\p_{v_iv_j}f-\bar c f}}\cdot\com{\p^\mu  f(t,v)}
 dv\\
 &=&2\sum_{1\leq i,j\leq 3}\int_{\Real^3}
  \bar a_{ij}\inner{\p_{v_iv_j}\p^\mu  f}\cdot\inner{\p^\mu  f } dv+{}\\
  &{}&+2\sum_{1\leq i,j\leq 3}\sum\limits_{\abs\beta=1}
  C_{\mu }^\beta\int_{\Real^3}\inner{\p^\beta\bar a_{ij}}\inner{\p_{v_iv_j}\p^{\mu -\beta} f}
  \cdot\inner{\p^\mu f } dv+{}\\
  &{}&+2\sum_{1\leq i,j\leq 3}\sum_{\abs\beta=2}^{\abs\mu}
  C_{\mu }^\beta\int_{\Real^3}\inner{\p^\beta\bar a_{ij}}\inner{\p_{v_iv_j}\p^{\mu -\beta} f}
  \cdot\inner{\p^\mu f } dv-{}\\
  &{}&-2\sum\limits_{0\leq\abs\beta\leq\abs\mu }
  C_{\mu }^\beta\int_{\Real^3}\inner{\p^{\mu-\beta}\bar c}\inner{
  \p^{\beta} f}\cdot\inner{\p^\mu  f }dv\\
  &=&(I)+(II)+(III)+(IV).
\end{eqnarray*}
We shall estimate the each term above by following steps.

\smallskip
 {\bf Step 1. Upper bound for the term $(I).$}

\smallskip
Integrating by parts, one has
\begin{eqnarray*}
  (I)&=&-2\sum_{1\leq i,j\leq 3}\int_{\Real^3}
  \bar a_{ij}\inner{\p_{v_j}\p^\mu f}\cdot\inner{\p_{v_i}\p^\mu f } dv-2\sum_{1\leq j\leq 3}\int_{\Real^3}
  \bar b_{j}\inner{\p_{v_j}\p^\mu f}\cdot\inner{\p^\mu f } dv\\
  &=&(I)_1+(I)_2.
\end{eqnarray*}
The ellipticity property \reff{ellipticity} of $(\bar a_{ij})_{i,j}$
gives that
\begin{eqnarray*}
  (I)_1\leq -2K\int_{\Real^3}
  \abs{\nabla_v\p^\mu f}^2(1+\abs v^2)^{\gamma/2}dv
  = -2K\norm{\nabla_v\p^\mu f(t)}_{L^2_\gamma}^2.
\end{eqnarray*}
For the term $(I)_2,$ we integrate by parts to get
\begin{eqnarray*}
  (I)_2=-(I)_2+2\int_{\Real^3}
  \bar c\inner{\p^\mu f}\cdot\inner{\p^\mu f } dv.
\end{eqnarray*}
This along with the fact
  \[\abs{\bar c(v)}\leq C\norm{f(t)}_{L^1_\gamma}
  (1+\abs v^2)^{\gamma/2}\leq C
  (1+\abs v^2)^{\gamma/2}
\]
 shows immediately
\[
  (I)_2\leq C\norm{\p^\mu f(t)}_{L^2_\gamma}^2
  \leq C\norm{\nabla_v\p^{\mu-1} f(t)}_{L^2_\gamma}^2.
\]
Combining these estimates, we get the upper bound for the term
$(I)$, i.e.
\begin{equation}\label{I}
  (I)\leq -2K\norm{\nabla_v\p^\mu f(t)}_{L^2_{\gamma}}^2+C\norm{\nabla_v\p^{\mu-1} f(t)}_{L^2_{\gamma}}^2.
\end{equation}

\smallskip
{\bf Step 2. Upper bound for the term $(II).$}
\smallskip

Recall $(II)=2\sum_{1\leq i,j\leq 3}\sum\limits_{\abs\beta=1} C_{\mu
}^\beta\int_{\Real^3}\inner{\p^\beta\bar
a_{ij}}\inner{\p_{v_iv_j}\p^{\mu -\beta} f} \cdot\inner{\p^\mu f
}dv.$ Integrating by parts, we get
\begin{eqnarray*}
 (II)&=&-2\sum_{1\leq j\leq 3}\sum\limits_{\abs\beta=1}
  C_{\mu}^\beta\int_{\Real^3}\inner{\p^\beta\bar b_{j}}\inner{\p_{v_j}\p^{\mu-\beta} f}
  \cdot\inner{\p^\mu f } dv-{}\\
  &{}&-2\sum_{1\leq i,j\leq 3}\sum\limits_{\abs\beta=1}
  C_{\mu}^\beta\int_{\Real^3}\inner{\p^\beta\bar a_{ij}}\inner{\p_{v_j}\p^{\mu-\beta} f}
  \cdot\inner{\p_{v_i}\p^\mu f } dv\\
  &=&(II)_1+(II)_2.
\end{eqnarray*}
Note $\abs{\p^\beta\bar b_{j}(t,v)}\leq C(1+\abs v^2)^{\gamma/2}$
for any $\beta$ with $\abs\beta=1$ and hence
\begin{eqnarray*}
 (II)_1\leq C\abs\mu\cdot\norm{\nabla_v\p^{\mu-\beta}f(t)}_{L^2_\gamma}
 \norm{\p^{\mu}f(t)}_{L^2_\gamma}
 \leq
 C\abs\mu\cdot\norm{\nabla_v\p^{\mu-1}f(t)}_{L^2_\gamma}^2.
\end{eqnarray*}
For the term $(II)_2$, since $\mu=\beta+(\mu-\beta),$ it can be
rewritten as following form
\begin{eqnarray*}
 (II)_2&=-2\sum_{1\leq i,j\leq 3}\sum\limits_{\abs\beta=1}
  C_{\mu}^\beta\int_{\Real^3}\inner{\p^\beta\bar a_{ij}}\inner{\p_{v_j}\p^{\mu-\beta} f}
  \cdot\inner{\p^\beta\p_{v_i}\p^{\mu-\beta} f } dv.
\end{eqnarray*}
Since $\abs \beta=1,$ we can integrate by parts to get
\begin{eqnarray*}
  (II)_2&=&2\sum_{1\leq i,j\leq 3}\sum\limits_{\abs\beta=1}
  C_{\mu}^\beta\int_{\Real^3}\inner{\p^{\beta}\bar a_{ij}}\inner{\p_{v_j}\p^{\mu} f}
  \cdot\inner{\p_{v_i}\p^{\mu-\beta} f } dv+{}\\
  &{}&+2\sum_{1\leq i,j\leq 3}\sum\limits_{\abs\beta=1}
  C_{\mu}^\beta\int_{\Real^3}\inner{\p^{\beta+\beta}\bar a_{ij}}\inner{\p_{v_j}\p^{\mu-\beta} f}
  \cdot\inner{\p_{v_i}\p^{\mu-\beta} f } dv\\
  &=&-(II)_2+2\sum_{1\leq i,j\leq 3}\sum\limits_{\abs\beta=1}
  C_{\mu}^\beta\int_{\Real^3}\inner{\p^{\beta+\beta}\bar a_{ij}}\inner{\p_{v_j}\p^{\mu-\beta} f}
  \cdot\inner{\p_{v_i}\p^{\mu-\beta} f } dv.
\end{eqnarray*}
Hence
\begin{eqnarray*}
  (II)_2&=\sum_{1\leq i,j\leq 3}\sum\limits_{\abs\beta=1}
  C_{\mu}^\beta\int_{\Real^3}\inner{\p^{\beta+\beta}\bar a_{ij}}\inner{\p_{v_j}\p^{\mu-\beta} f}
  \cdot\inner{\p_{v_i}\p^{\mu-\beta} f } dv.
\end{eqnarray*}
This along with the fact $\abs{\p^{\beta+\beta}\bar a_{ij}(v)}\leq C
(1+\abs v^2)^{\gamma/2}$ for all $\beta$ with $\abs\beta=1$ gives
that
\begin{eqnarray*}
 (II)_2\leq C\sum\limits_{\abs\beta=1}
  C_{\mu}^\beta\cdot\norm{\nabla_v\p^{\mu-\beta}f}_{L^2_{\gamma}}^2
  \leq C\abs\mu\cdot\norm{\nabla_v\p^{\mu-1}f}_{L^2_{\gamma}}^2.
\end{eqnarray*}
Thus we obtain
\begin{equation}\label{II}
 (II)\leq C\abs\mu\cdot\norm{\nabla_v\p^{\abs\mu-1}f}_{L^2_{\gamma}}^2.
\end{equation}

\smallskip
{\bf Step 3. Upper bound for the terms $(III)$ and $(IV)$ and the
conclusion.}
\smallskip

Recall $(III)=2\sum_{1\leq i,j\leq 3}\sum_{\abs\beta=2}^{\abs\mu}
  C_{\mu }^\beta\int_{\Real^3}\inner{\p^\beta\bar a_{ij}}\inner{\p_{v_iv_j}\p^{\mu -\beta} f}
  \cdot\inner{\p^\mu f }\,dv$, and
\[
  (IV)=-2\sum\limits_{0\leq\abs\beta\leq\abs\mu }
  C_{\mu }^\beta\int_{\Real^3}\inner{\p^{\mu -\beta}\bar c}\inner{
  \p^\beta f}\cdot\inner{\p^\mu  f }dv.
\]
By virtue of Lemma \ref{a} and lemma \ref{c}, it follows that
\begin{eqnarray}\label{III}
  (III)\leq C \sum_{i,j=1}^3\sum_{\abs\beta=2}^{\abs\mu}
  C_{\mu }^\beta \norm{\p_{v_iv_j}\p^{\mu -\beta}f(t)}_{L^2_\gamma}
  \cdot\norm{\p^\mu f(t)}_{L^2_\gamma}\com{G(f(t))}_{\beta-2}\nonumber\\
  \leq C \sum_{\abs\beta=2}^{\abs\mu}
  C_{\mu }^\beta \norm{\nabla_v\p^{\mu -\beta+1} f(t)}_{L^2_\gamma}
  \cdot\norm{\nabla_v\p^{\mu-1}
  f(t)}_{L^2_\gamma}\com{G(f(t))}_{\beta-2},
\end{eqnarray}
and
\begin{eqnarray}\label{IV}
  (IV)\leq C\sum\limits_{0\leq\abs\beta\leq\abs\mu}
  C_{\mu }^\beta\norm{
  \p^\beta f(t)}_{L^2_\gamma}\cdot
  \norm{\nabla_v\p^{\mu-1}  f(t) }_{L^2_\gamma}\cdot\com{G(f(t))}_{{\mu}-\beta}.
\end{eqnarray}
Combining the estimates \reff{I}-\reff{IV}, we can deduce the
desired estimate in Lemma \ref{crucial}, that is
\begin{eqnarray*}
  &&\frac{d}{dt}\norm{\p^\mu f(t)}_{L^2}^2+C_1\norm{\nabla_v\p^\mu  f(t) }_{L^2_\gamma}^2
  \leq C_2\abs\mu^2\norm{\nabla_v\p^{\mu-1} f(t)}_{L^2_\gamma}^2+{}\\
  &{}&~
~+C_2 \sum\limits_{2\leq\abs\beta\leq\abs\mu}
  C_{\mu }^\beta \norm{\nabla_v\p^{\mu -\beta+1} f(t)}_{L^2_\gamma}
  \norm{\nabla_v\p^{\mu-1}
  f(t)}_{L^2_\gamma}\cdot\com{G(f(t))}_{\beta-2}+{}\\
  &{}&~~+C_2\sum\limits_{0\leq\abs\beta\leq\abs\mu}
   C_{\mu }^\beta\norm{
  \p^\beta f(t)}_{L^2_\gamma}\cdot\norm{\nabla_v\p^{\mu-1}  f(t)
  }_{L^2_\gamma}\cdot\com{G(f(t))}_{{\mu}-\beta},
\end{eqnarray*}
where $C_1, C_2$ are two constants depending only on $M_0, E_0, H_0$
and $\gamma.$  This completes the proof of Lemma \ref{crucial}.

\end{proof}

\end{document}